\documentclass[12pt,reqno,oneside]{amsart}
      
\usepackage{amsmath,amsthm,amsfonts,amssymb}
\usepackage{indentfirst}
\usepackage{url}
\usepackage{hyperref}
\usepackage{graphicx}

\theoremstyle{plain}
\newtheorem{theorem}{Theorem}[section]

\newtheorem{lem}[theorem]{Lemma}
\newtheorem{prop}[theorem]{Proposition}

\theoremstyle{definition}
\newtheorem{defn}[theorem]{Definition}
\newtheorem{exa}[theorem]{Example}
\newtheorem{obs}[theorem]{Remark}

\numberwithin{equation}{section}

\newcommand{\bbN}{\mathbb{N}}

\newcommand{\bbE}{\mathbb{E}}

\newcommand{\mcu}{C^1(p,r,\lambda)}
\newcommand{\mcd}{C^2(p,r,\lambda)}
\newcommand{\mci}{C^i(p,r,\lambda)}
\newcommand{\mct}{C^3(p,r,\lambda,m)}

\newcommand{\ld}{\lambda^2(p,r)}
\newcommand{\li}{\lambda^i(p,r)}
\newcommand{\lt}{\lambda^3(p,r,m)}

\begin{document}

\baselineskip=18pt

\title{Dispersion as a survival strategy}



\author[Valdivino Vargas Junior]{Valdivino Vargas Junior}
\address[Valdivino Vargas Junior]{Institute of Mathematics and Statistics, Federal University of Goias, Campus Samambaia, 
CEP 74001-970, Goi\^ania, GO, Brazil}
\email{vvjunior@ufg.br}
\thanks{Valdivino Vargas was supported by PNPD-CAPES (1536114), F\'abio Machado by CNPq (310829/2014-3) and Fapesp (09/52379-8) and Alejandro Roldan by CNPq (141046/2013-9).}

\author[F\'abio P. Machado]{F\'abio~Prates~Machado}
\address[F\'abio P. Machado]{Statistics Department, Institute of Mathematics and Statistics, University of S\~ao Paulo, CEP 05508-090, S\~ao Paulo, SP, Brazil.}
\email{fmachado@ime.usp.br}

\author[Alejandro Rold\'an]{Alejandro~Rold\'an-Correa}
\address[Alejandro Rold\'an]{Instituto de Matem\'aticas, Universidad de Antioquia, Calle 67, no 53-108, Medellin, Colombia}
\email{alejandro.roldan@udea.edu.co}


\date{\today}

\begin{abstract}
We consider stochastic growth models to represent population subject to catastrophes. 
We analyze the subject from different set ups considering or not spatial restrictions, whether 
dispersion is a good strategy to increase the population viability. We find out it strongly depends 
on the effect of a catastrophic event, the spatial constraints of the 
environment and the probability that each exposed individual survives when a disaster strikes.
\end{abstract}

\maketitle

\section{Introduction}
\label{S: Introduction}

Biological populations are often exposed to catastrophic events that cause mass extinction: Epidemics, natural disasters, etc. When mild versions of these disasters occur, survivors may develop strategies to improve the odds of their species survival. Some populations adopt dispersion as a strategy. 
Individuals of these populations disperse, trying to make new colonies that may succeed settling down depending on the new environment they encounter. Recently, Schinazi~\cite{S2014} and Machado \textit{et al.}~\cite{MRS2015} proposed stochastic models for this kind of population dynamics. For these models they concluded that dispersion is a good survival strategy. Earlier, Lanchier~\cite{Lanchier} considered the basic 
contact process on the lattice modified so that large
sets of individuals are simultaneously removed, which also models catastrophes. In this work there are qualitative
results about the effect of the shape of those sets on the survival of the process, with interesting
non-monotonic results, and dispersion is proved to be a better strategy in some contexts.

Moreover, Brockwell \textit{et al.}~\cite{BGR1982} and later Artalejo \textit{et al.}~\cite{AEL2007} considered a model for the growth of a population (a single colony) subject to collapse. In their model, two types of effects when a disaster strikes were analyzed separately, \textit{binomial effect} and \textit{geometric effect}. After the collapse, the survivors remain together in the same colony (there is no dispersion). They carried out an extensive analysis including first extinction time, number of individuals removed, survival time of a tagged individual, and maximum population size reached between two consecutive extinctions. For a nice literature overview and motivation see Kapodistria \textit{et al.}~\cite{KPR2016}.

Based on the model proposed by Artalejo \textit{et al.}~\cite{AEL2007}, and adapting some ideas from Schinazi~\cite{S2014} and Machado~\textit{et al.}~\cite{MRS2015}, we analyze growth models of populations subject to disasters, where after the collapse species adopt dispersion as a survival strategy. We show that dispersion is not always a good strategy to avoid the population extinction. It strongly depends on the effect of a catastrophic event, the spatial constraints of the environment and the probability that each exposed individual survives when a disaster strikes.

This paper is divided into four sections. In Section 2 we define and characterize three models for the growth of populations subject to collapses. In Section 3 we compare the three models introduced in Section 2 and determine under what conditions the dispersion is a good strategy for survival, due to space restrictions and the effects when a disaster strikes. Finally, in Section 4 we prove the results from Sections 2 and 3.

\section{Growth models}

First we describe a model presented in Artalejo \textit{et al.}~\cite{AEL2007}. This is a model for a population which sticks together in one colony, without dispersion.  The colony gives birth to a new individual at rate $\lambda>0$,  while collapses happen at rate $\mu$. If at a collapse time the size of the population is $i$, it is reduced to $j$ with probability $\mu_{ij}$. The parameters $\mu_{ij}$ are determinated by how the collapse affects the population size. 
Next we describe two types of effects.

$\bullet$ \textit{Binomial effect:} Disasters reach the individuals simultaneously and independently of everything else. Each individual survives with probability $p<1$ (dies with probability $q=1-p$), meaning that 
\[ \mu_{ij}^B ={ i \choose j} p^j q^{i-j}, \ 0\leq j\leq i.\]

$\bullet$  \textit{Geometric effect:} Disasters reach the individuals sequentially and the effects of a disaster stop as soon as the first individual survives, if there are any survivor. The probability of next individual 
to survive given that everyone fails up to that point is $p<1,$ which
means that   
\[ \mu_ {ij}^G = \left\{\begin{array}{ll} q^i, & j=0\\ pq^{i-j}, & 1\leq j \leq i.\end{array}\right.\]

The binomial effect is appropriate when the catastrophe affects the individuals in a independent 
and even way. The geometric effect would correspond to cases where
the decline in the population is halted as soon as
any individual survives the catastrophic event. This
may be appropriate for some forms of catastrophic
epidemics or when the  catastrophe has a sequential propagation effect like in the predator-prey models - the predator kills prey until it becomes satisfied. More examples can be found in Artalejo \textit{et al.}~\cite{AEL2007} and in Cairns and Pollett~\cite{CairnsPollett}.

%

\subsection{Growth model without dispersion} 

In Artalejo \textit{et al.}~\cite{AEL2007} the authors consider the binomial and the geometric 
effect separately as alternatives to the total catastrophe rule which instantaneously 
removes the whole population whenever a catastrophic event occurs. 

Here we consider a mixture of both effects, that is, with probability $r$ the group is striken sequentially (geometric effect) and with probability $1-r$ the group is striken 
simultaneously (binomial effect). More precisely, 
\[ \mu_{ij}:=r\mu_{ij}^G+(1-r)\mu_{ij}^B.\]

We assume that the collapse rate $\mu$ equals 1. The size of the population (number of individuals in the colony) at time $t$ is a continuous time Markov process $\left\{X(t):t\geq 0\right\}$ whose infinitesimal generator $(q_{ij})_{i,j\geq 0}$ is given by 
\[q_{ij}=\left\{\begin{array}{ll}\lambda, & j=i+1, \ i\geq 0, \\
\mu_{ij}, & 0\leq j <i, \\
-(\lambda+\sum_{j=0}^{i-1}\mu_{ij}), & i=j, \\
0& \text{otherwise.} \end{array}\right.\]

We also assume $X(0)=1$ and denote by $\mcu$ the process described by $\left\{X(t):t\geq 0\right\}$. When $r=0$ and $r=1$, we obtain the models considered in Artalejo \textit{et al.}~\cite{AEL2007}.

\begin{theorem}[Artalejo \textit{et al.}~\cite{AEL2007}]
\label{th:semdisp}
Let $X(t)$ a process $\mcu$, with $\lambda>0$ and   $0<p<1$. Then, extinction (which means  $X(t)=0$ for some $t>0$)  occurs with probability
$$\rho_1(r)=\left\{\begin{array}{ll}
1& \text{, when } r<1\\
\min\left\{\frac{1-p}{\lambda p}, 1\right\} & \text{, when} r=1.
\end{array}\right.$$
Moreover, if $r<1$, or $r=1$ and $\lambda p < 1-p,$ the time it takes until extinction has finite expectation. 
\end{theorem}

\begin{obs} 
The result of Theorem~\ref{th:semdisp} has been shown by Artalejo \textit{et al.}~\cite{AEL2007} for the cases $r=0$ and $r=1$. 
They use the word \textit{extinction} to describe the event that 
$X(t) = 0$, for some $t>0$, for a process where state 0 is not 
an absorbing state. In fact the extinction time here is the first hitting time to the state 0. We keep using the word extinction for this model trough the paper.

From their result one can see that survival is only possible when the effect is purely geometric ($r=1$). The reason for that
is quite clear: If $r<1$ the binomial effect strikes at rate $(1-r)>0$ so even if one
considers $p=1$ when the geometric effect strikes, the population will die out as proved
in Artalejo \textit{et al.}~\cite{AEL2007} for the case $r=0$.
\end{obs}

\subsection{Growth model with dispersion but no spatial restriction.} 
Consider a population of individuals divided into separate colonies. Each colony begins with  
an individual. The number of individuals in each colony increases independently according 
to a Poisson process of rate $\lambda > 0 $. Every time an exponential time of mean 1 occurs, the colony 
collapses through a binomial or a geometric effect and each of the collapse survivors 
begins a new colony independently of everything else.
We denote this process by $\mcd$ and consider it starting from a single colony with just one
individual.

The following theorem establishes necessary and sufficient conditions for survival in
$\mcd.$

\begin{theorem}\label{th:disp1} 
The process $\mcd$ survives with positive
probability if and only if
\begin{equation}\label{eqthdisp1}
 \frac{p(\lambda+1)^2r}{\lambda p+1} +p(\lambda+1)(1-r)>1.
\end{equation}
\end{theorem}

Theorem~\ref{th:disp1} shows that, contrary to what happens in $\mcu$, in $\mcd$ the 
population is able to survive even when the binomial effect may occur $(r<1)$. See
example~\ref{ex:bin}. In particular, if $r=0$ (pure binomial effect)
the process survives with positive probability whenever $p(\lambda+1)>1$.

The next result shows how to compute the probability of extinction, which means, the probability
that eventually the system becomes empty.

\begin{theorem}\label{th:disp2} Let $\rho_2(r)$ be the probability of extinction in $\mcd$. Then $\rho_2(r)$  is the smallest non-negative solution of 
\begin{equation}\label{probext}\phi(s):=\frac{1}{1+\lambda p}\left[q+\frac{r(\lambda +1)ps}{1+\lambda -\lambda s}+\frac{(1-r)(\lambda +1)ps}{1+\lambda p - \lambda p s}\right]=s 
\end{equation}
\end{theorem}

\begin{exa}\label{ex:bin} For $C^2(2/5,r,1)$
\[\phi(s)=\frac{3}{7}+\frac{4rs}{14-7s}+\frac{20(1-r)s}{49-14s}.\]
The smallest non-negative solution for the equation $\phi(s)=s,$ is given by 
\[\rho_2(r)=\left\{\begin{array}{cl}1,& \ r\leq7/12\\ 
\displaystyle\frac{12r+49-\sqrt{144 r^2+1176 r+49}}{28},& \ r>7/12.
\end{array}\right.\]
\end{exa}

\begin{obs} For $r=0$ (pure binomial effect) and  $r=1$ (pure geometric effect) the smallest non-negative solution for (\ref{probext}) is: 
$$\rho_2(0)=\min\left\{\frac{q}{\lambda p},1\right\} \hspace{0.5cm} \text{and}\hspace{0.5cm}\rho_2(1)=\min\left\{\frac{q(\lambda+1)}{\lambda(1+\lambda p)},1\right\}.$$ 

Observe that $\rho_2(0)\geq \rho_2(1)$ where the strict inequality holds provided $(1+\lambda+\lambda^2)^{-1}<p<1.$ Moreover, 
\begin{itemize}
\item[$\bullet$] If $p<\displaystyle\frac{1}{1+\lambda+\lambda^2}$ then $\rho_2(0)=\rho_2(1)=1.$
\item[$\bullet$] If $\displaystyle\frac{1}{1+\lambda+\lambda^2}<p<\frac{1}{1+\lambda}$ then 
$\rho_2(0)=1$ and $\rho_2(1)=\displaystyle\frac{q(\lambda+1)}{\lambda(1+\lambda p)}.$
\item[$\bullet$] If $p>\displaystyle\frac{1}{1+\lambda}$ then $\rho_2(0)=\displaystyle\frac{q}{\lambda p}$ and $\rho_2(1)=\displaystyle\frac{q(\lambda+1)}{\lambda(1+\lambda p)}.$
\end{itemize}

Note that likewise as occurs in $\mcu$, the binomial effect is a worst scenary than the
geometric effect for the population survival in $\mcd$.
\end{obs}

\begin{obs} Observe that $\rho_1(r) \geq \rho_2(r)$ for $ r \in [0,1].$ In addition, if $r<1$ the inequality is strict
provided (\ref{eqthdisp1}) holds. Moreover,  
$\rho_1(1) > \rho_2(1)$ for $\lambda(1+\lambda p)>q(\lambda+1).$ That means when there
are no spatial restrictions, dispersion is a good strategy for population survival. That coincides
with the results for the models presented and analyzed by Schinazi~\cite{S2014} and Machado \textit{et al.}~\cite{MRS2015}. \end{obs}

\subsection{Growth   with dispersion and spatial restriction.} 

 
Let $\mathcal{G}_m$ be a graph (finite or infinite) such that every vertex has $m$ neighbours, what is known as a $m-$regular graph. Let us define a process with 
dispersion and spatial restrictions on $\mathcal{G}_m$, starting from a single colony placed at one 
vertex of $\mathcal{G}_m$, with just one individual. The 
number of individuals in a colony grows following a Poisson process of rate $\lambda>0$. To each colony 
we associate an exponential time of mean 1 that indicates when the colony collapses. Each one of the individuals that survived the collapse (either a binomial or a geometric effect)
picks randomly a neighbor vertex and tries to create a new colony at it. Among the survivors leaping 
to the same vertex trying to create a new colony at it, only one succeeds (disregarding the number of
colonies already present at that vertex), the others die. So in this case  when a colony collapses, 
it is replaced by 0,1, ... or $ m $ colonies. Finaly, every vertex can have any number of independent
colonies. We denote this process by $\mct$.

The next result presents a necessary and sufficient condition for population survival in
$\mct$. 

\begin{theorem} \label{th:dispesp1}
The process $\mct$ survives with positive probability if and only if $$\frac{mp(1+\lambda)^2r}{(m+\lambda)(\lambda p +1)}+\frac{mp(1+\lambda)(1-r)}{m+ \lambda  p} > 1.$$
\end{theorem} 
The following result shows that the extinction probability for the process $\mct$
can be computed as the root of a polynomial of degree $m$.


\begin{theorem}\label{th:dispesp2} Let $\rho_3(r)$ be the probability of population extinction in $\mct$. Then  $\rho_3(r)$  is the smallest non-negative solution of $$\psi(s):=r\psi _G(s)+(1-r)\psi _B(s)=s,$$ where
{\small  
\[ \psi _B(s):=\frac{q}{1+\lambda p}+\frac{m(1+\lambda)}{\lambda}\sum_{k=1}^m  {m \choose k}\left[\frac{-\lambda p s}{m(1+\lambda p)}\right]^k\sum_{j=0}^k {k \choose j}\frac{(-1)^j j^k}{m(1+\lambda p)-\lambda p j}, \]}  
{\small 
\[ \psi_G(s):=\frac{q}{1+\lambda p}+\frac{(1+\lambda)ps}{\lambda p +1}\sum_{k=1}^m  {m \choose k}\left[\frac{-\lambda  s}{m(1+\lambda )}\right]^{k-1}\sum_{j=0}^k {k \choose j}\frac{(-1)^{j-1}j^k}{m(1+\lambda )-\lambda j}.\]}  
\end{theorem}

\begin{exa} Consider $C^3(2/3,r,1,3).$ Then
$$\psi(s)=\left(\frac{126 r}{3575}+\frac{32}{715}\right)s^3
+\left(\frac{138r}{3575}+\frac{144}{715}\right)s^2
+\left(\frac{36}{65}-\frac{24r}{325}\right)s
+\frac{1}{5}.$$
Therefore, the smallest non-negative solution for $\psi(s)=s$ is given by
$$\rho_3(r)=\frac{-440-132 r+\sqrt{22(14000+9375 r+792 r^2)}}{2 (80+63 r)}.$$ 
\end{exa}

\section{Dispersion as a survival strategy}

Towards being able to evaluate dispersion as a survival strategy we define
$$\li:=\inf\{\lambda:  \mathbb{P}[ \mci \text{ survives}]>0 \}, \quad \text{for } i=1,2$$ 
$$\text{and} \quad \lt:=\inf\{\lambda:  \mathbb{P}[ \mct \text{ survives}]>0 \}. $$

Observe that for $i=1,2$, when $0<\li<\infty$ for $0<p<1,$ the graph 
of $\li$ splits the parametric space $\lambda \times p$ into two regions. For those values 
of $(\lambda,p)$ above the curve $\li$  there is survival in $\mci$ with positive probability,
and for those values of $(\lambda,p)$  below the curve  $\li$ extinction occurs in $\mci$ 
with probability 1. The analogous happens also for i=3 and any $m$.

Next we establish some properties of $\ld$ and $\lt.$  
\begin{prop}\label{prop-disp-est} Let $0\leq r \leq 1$ and $0<p<1.$ Then,
\begin{itemize}
\item[$(i)$] $0 < \ld < \lambda^3(p,r,m+1) < \lt < \infty,$ for all $ m\geq 2.$ Besides $\lambda^3(p,r,1)=\infty.$
\item[$(ii)$] $\displaystyle\lim_{m\rightarrow\infty}\lt=\ld.$
\end{itemize} 
\end{prop}

\begin{obs} From standard coupling arguments one can show the expected monotonocity relationship.

\noindent
If $p_1 > p_2$ then
\begin{align*}
\lambda^i(p_1, r) & \leq \lambda^i(p_2, r), \ i=1,2 \\
\lambda^3(p_1, r,m) & \leq \lambda^3(p_2, r,m).
\end{align*}
If $r_1 > r_2$ then
\begin{align*}
\lambda^i(p, r_1) & \leq \lambda^i(p, r_2), \ i=1,2 \\
\lambda^3(p, r_1,m) & \leq \lambda^3(p, r_2,m).
\end{align*}
\end{obs}

For what follows $0<p<1.$ From Theorem \ref{th:semdisp} it follows that if $r<1$ then $\lambda^1(p,r)=\infty,$ and from Proposition \ref{prop-disp-est} we obtain that $$\ld<\lt<\lambda^1(p,r),$$ for all $m\geq2$. Then, 
provided binomial effect may strike $(r<1)$, dispersion is a good  scenary for 
population survival either with or without spatial restrictions.  \\

When binomial effect is not present $(r=1)$, which means, only  geometric effect is present,  it
is simple to compute $\lambda^1(p,1),$ $\lambda^2(p,1)$ and $\lambda^3(p,1,m)$. From Theorems \ref{th:semdisp}, \ref{th:disp1} and \ref{th:dispesp1},   we have that 

%
%

\begin{eqnarray*}
\lambda^1(p,1)&=&\frac{1-p}{p},\\
\lambda^2(p,1)&=&\sqrt{\frac{1}{4}+\frac{1-p}{p}} -\frac{1}{2},\\
\lambda^3(p,1,m)&=&\frac{1-mp+\sqrt{(1-mp)^2+4m(m-1)p(1-p)}}{2p(m-1)}.
\end{eqnarray*}

\noindent 
When $r=1$ (pure geometric effect) $\lambda^2(p,1) < \lambda^1(p,1).$ However, dispersion is not always a better scenary for
population survival, as one can see in Figure~\ref{fig:sub1}. Observe that \[\lambda^3(p,1,m)\leq\lambda^1(p,1) \iff p\leq 1-\frac{1}{m-1}.\]
Therefore, under a pure geometric effect, dispersion is an advantage or not for population 
survival depending on both $m$, the spatial restrictions, and $p$, the probability that an individual, 
when exposed to catastrophe, survives. See Figure \ref{fig:sub2}.  

\begin{figure}[h!]
\centering
 \includegraphics[width=\textwidth]{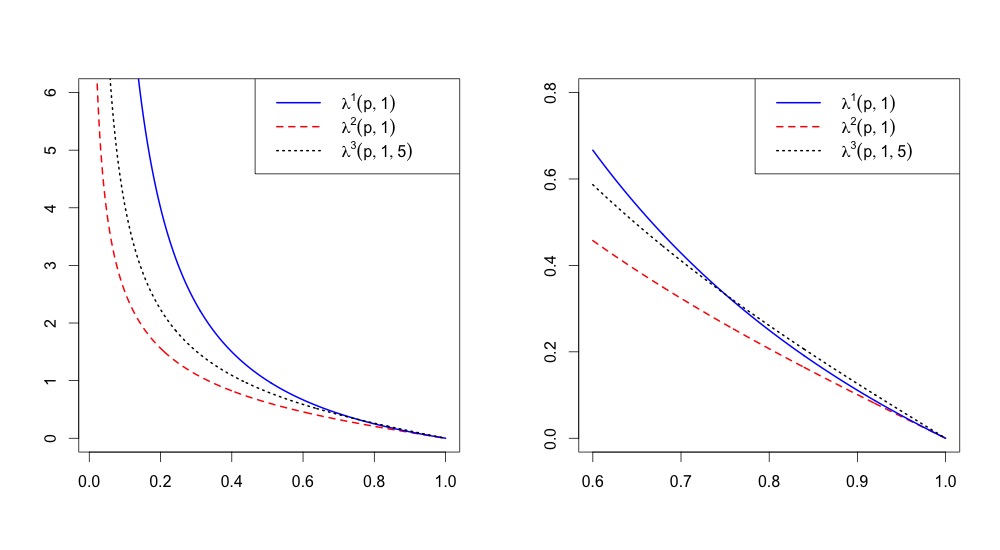}
 \caption{Graphics of $\lambda^1(p,1),\lambda^2(p,1), \lambda^3(p,1,5)$ }
 \label{fig:sub1}
\end{figure}

\begin{figure}[h!]
\centering 
 \includegraphics[scale=0.60]{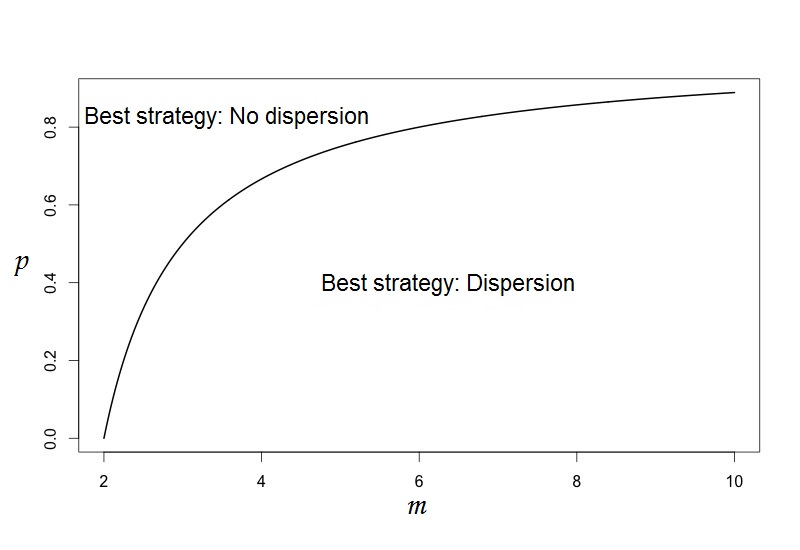}
 \caption{Curve $p=1-(m-1)^{-1}.$ Best strategy for survival, when r=1, provided the spatial restrictions $(m)$
and the probability that an individual survives when facing a collapse $(p)$.}
 \label{fig:sub2}
\end{figure}

\section{Proofs}

Theorem~\ref{th:semdisp} is part of Theorem 3.1 and Theorem 3.2 in Artalejo 
\textit{et al.}~\cite{AEL2007}. They work hard with the moment generating functions of the
first excursion until 0 (the empty state) when the process (binomial and geometric catastrophes)
starts from 1 individual. Here we 
present an alternative proof for $r<1$ by the use of Foster's theorem, enunciated next. For a proof of Foster's theorem see Fayolle~\textit{et. al. }\cite[Theorem 2.2.3]{FMM1995}. 
 
\begin{theorem}[Foster's theorem] Let $\{W_n\}_{n\geq 0}$  be an irreducible and aperiodic Markov chain on countable state space $\mathcal{A}=\{\alpha_i,\ i\geq0\}.$ Then, $\{W_n\}_{n\geq 0}$ is ergodic if and only if there exists a positive function $f(\alpha), \ \alpha\in\mathcal{A},$ a number $\epsilon>0$ and a finite set $A\subset\mathcal{A}$ such that 
$$\mathbb{E}[f(W_{n+1})-f(W_{n}) \ | \ W_n=\alpha_j]\leq -\epsilon, \quad \alpha_j\notin A,$$ 
$$\mathbb{E}[f(W_{n+1}) \ | \ W_n=\alpha_i] < \infty, \quad \alpha_i\in A.$$
\end{theorem}

Next we present the proof of Theorem~\ref{th:semdisp}.

\begin{proof}[Proof of Theorem \ref{th:semdisp}]

Let $\{Y_n\}_{n\geq 0}$ be a discrete-time Markov chain embedded on $\mcu,$ with transition 
probabilities given by
$$\begin{array}{ll}
P_{i,i+1}=\displaystyle\frac{\lambda}{\lambda +1},&  \ i\geq 0, \\ \\
P_{i,j}=\displaystyle\frac{r\mu_{ij}^G+(1-r)\mu_{ij}^B}{\lambda +1}, &  \  0\leq j\leq i.\\
\end{array}$$

Ergodicity of $\{Y_n\}$ implies that the time until extintion of $\mcu$ has finite mean.

Observe that $\{Y_n\}$ is irreducible and aperiodic. We use Foster's theorem to show that $\{Y_n\}_{n\geq 0}$ is ergodic for $0 \leq r<1$, $0<p<1$ 
and $\lambda>0$. Consider the function
$f:\mathbb{N}\rightarrow \mathbb{R}^+$ defined by $f(i)=i+1$,  $\epsilon>0$ and the set
$$A:=\left\{i\in \mathbb{N}: \frac{\lambda-i(1-r)q}{1+\lambda}-\frac{rq(1-q^{i})}{p(1+\lambda)} >-\epsilon\right\}.$$ 

For $0 \leq r<1,$ $0<p<1$ and $\lambda>0,$ the set $A$ is finite. Moreover we have that

\noindent $\begin{array}{lll}
\bullet \ \mathbb{E}[f(Y_{n+1})&-&f(Y_{n}) \ | \ Y_n=i]=\displaystyle\sum_{j=0}^{i+1}[f(j)-f(i)]P_{i,j}\\ \\
&=& \displaystyle\frac{\lambda}{1+\lambda}+\sum_{j=0}^i (j-i)\left[       \frac{r\mu_{ij}^G+(1-r)\mu_{ij}^B}{1+\lambda}\right]\\ \\
&=&\displaystyle\frac{\lambda}{1+\lambda}+\frac{1}{1+\lambda}\left[-riq^i+r\sum_{j=1}^i (j-i)pq^{i-j} \right.
\\ \\
&& \ \left. + \ (1-r)\displaystyle\sum_{j=0}^i(j-i){i\choose j}p^jq^{i-j}   \right]\\ \\
&=& \displaystyle\frac{\lambda-i(1-r)q}{1+\lambda}-\frac{rq(1-q^{i})}{p(1+\lambda)}\\ \\
&\leq & -\epsilon \quad \text{for } i\notin A.
\end{array}$\\ \\

\noindent$\begin{array}{lll}
\bullet \ \mathbb{E}[f(Y_{n+1}) &\ |& \ Y_n=i] = \displaystyle\sum_{j=0}^{i+1}f(j)P_{i,j} \le (i+2)^2 < \infty \text{ for } i\in A. 
\end{array}$\\

%

\noindent It follows from Foster's theorem that $\{Y_n\}$ is ergodic and that concludes the proof.
\end{proof}

Seeking the proof of the other results we define the following auxiliary process. \\

\noindent
\textbf{Auxiliary process $(Z_n^{r,i})_{n\geq 0}$}:

\noindent 
Consider $\mcd$ and $\mct$. We define $Z_0^{r,i}=1$ for $i=2,3$, the number 
of colonies present at time 0 in each model. As soon as it collapses, $Z_1^{r,i}$, a random number 
of colonies will be created, the first generation. Each one of these colonies will give birth (at 
different times) to a random number of new colonies, the second generation. Let us define
this quantity by $Z_2^{r,i}$. In general, for $n \geq 1$, if $Z_{n-1}^{r,i} = 0$ then 
$Z_n^{r,i}=0$. On the other hand, if
$Z_{n-1}^{r,i} \geq 1$ then $Z_n^{r,i}$  is the number of colonies generated by the $(n-1)-th$ 
generation of colonies. 

From the fact that the numbers of descendants of different colonies are independent and have
the same distribution, we claim that $\{Z_n^{r,i}\}_{n \in \bbN}$ is a Galton-Watson process.

\begin{obs}\label{AuxPro} For $i=2,3,$ observe that  $\mci$ dies out if and only if  
$\{Z_n^{r,i}\}_{n \in \bbN}$ dies out, which in turn happens almost surely if and only if $\mathbb{E}[Z_1^{r,i}] \leq 1$. The probability
of extinction for  $\{Z_n^{r,i}\}_{n \in \bbN}$ is the smallest non-negative solution of 
$\phi_{r,i}(s)=s,$ 
where $\phi_{r,i}(s)$ is the probability generating function of $Z_1^{r,i}$. 
\end{obs}

\begin{lem}\label{L:disp}  The probability generating function of $Z_1^{r,2}$ is given by:
\[ \phi_{r,2}(s)=\frac{1}{1+\lambda p}\left[q+\frac{r(\lambda +1)ps}{1+\lambda -\lambda s}+\frac{(1-r)(\lambda +1)ps}{1+\lambda p - \lambda p s}\right] \]
and
\[ \mathbb{E}[Z_1^{r,2}]=\frac{p(\lambda+1)^2r}{\lambda p+1} +p(\lambda+1)(1-r).\]
\end{lem}

\begin{proof}
$Z_1^{r,2}$ is the number of colonies in the first generation of $\mcd.$ 
Denote $Z_B:=Z_1^{0,2}$ and $Z_G:=Z_1^{1,2}.$ Firstly we show that

\begin{eqnarray}\label{E1:lemaaux1}
\mathbb{P}[Z_B=k]&=&\left\{\begin{array}{ll} \displaystyle\frac{1+\lambda}{\lambda (1+\lambda p)}\left(\frac{\lambda p}{1+\lambda p}\right)^ k,& k\geq 1 \vspace{0,2cm} \\
\displaystyle\frac{q}{1+\lambda p}, & k=0.\end{array}\right.
\end{eqnarray} 
 \begin{eqnarray}\label{E2:lemaaux1}
 \mathbb{P}[Z_G=k]&=&\left\{\begin{array}{ll} \displaystyle\frac{p}{1+\lambda p}\left(\frac{\lambda }{1+\lambda }\right)^ {k-1},& k\geq 1 \vspace{0,2cm}\\
\displaystyle\frac{q}{1+\lambda p}, & k=0.\\\end{array}\right.
\end{eqnarray}

\begin{defn} Let us consider the following random variables

\begin{itemize}
\item $T$ the lifetime of the collony until the collapse time;
\item $f_T(t)$ the density of the random variable T;
\item $X_T$ the amount of individuals created in a collony until it collapes.
\end{itemize}
\end{defn}

Observe that
\begin{eqnarray}
\label{eq: ZBEZG}
\mathbb{P}[Z_B=k]=\int_0^{\infty} f_T(t) \sum_{n= 0 \vee k-1}^{\infty} \mathbb{P}(X_T=n|T=t)\mathbb{P}(Z_B=k|X_T=n; T=t)dt.
\end{eqnarray}

Then, for $k=0$, we have that
\[ \mathbb{P}[Z_B=0]=\int_0^\infty e^{-t} \sum_{n=0}^\infty \frac{e^{-\lambda t}(\lambda t)^n}{n!} q^{n+1} dt=q\int_0^\infty e^{-(\lambda p+1)t}  dt=\frac{q}{1+\lambda p}. \]

For $k\geq 1,$ 
\begin{eqnarray*}
\mathbb{P}[Z_B=k]&=&\displaystyle\int_0^\infty e^{-t} \sum_{n=k-1}^\infty \frac{e^{-\lambda t} (\lambda t)^n}{n!}{n+1 \choose k} p^kq^{n+1-k} dt   \\  
&=&q\left(\displaystyle\frac{p}{q}\right)^k \displaystyle\sum_{n=k-1}^\infty {n+1 \choose k} \frac{(\lambda q)^n}{n!} \int_0^\infty e^{-(\lambda +1)t} \  t^n dt  \\
&=&q\left(\displaystyle\frac{p}{q}\right)^k \displaystyle\sum_{n=k-1}^\infty {n+1 \choose k} \frac{(\lambda q)^n}{n!} \frac{\Gamma(n+1)}{(\lambda +1)^{n+1}} \\
&=&\displaystyle\frac{q}{\lambda +1} \left(\frac{p}{q}\right)^k \displaystyle\sum_{n=k-1}^\infty {n+1 \choose k} \left(\frac{\lambda q}{\lambda +1}\right)^{n} \\ 
&=&\displaystyle\frac{q}{\lambda +1} \left(\frac{p}{q}\right)^k \left(\frac{\lambda q}{\lambda +1}\right)^{k-1}\displaystyle\sum_{j=0}^\infty {j+k \choose k} \left(\frac{\lambda q}{\lambda +1}\right)^{j} \\
&=&\displaystyle\frac{q}{\lambda +1} \left(\frac{p}{q}\right)^k \left(\frac{\lambda q}{\lambda +1}\right)^{k-1} \left(1-\frac{\lambda q}{\lambda +1}\right)^{-(k+1)}\\
&=&\displaystyle\frac{1+\lambda}{\lambda (1+\lambda p)}\left(\frac{\lambda p}{1+\lambda p}\right)^ k.
\end{eqnarray*}

Similarly to~(\ref{eq: ZBEZG}), we obtain the distribution of $Z_G$. First observe that
$\mathbb{P}[Z_B=0]=\mathbb{P}[Z_G=0]$. 
Besides, for $k\geq 1,$ 
\begin{eqnarray*}
\mathbb{P}[Z_G=k]&=&\displaystyle\int_0^\infty e^{-t} \sum_{n=k-1}^\infty \frac{e^{-\lambda t}(\lambda t)^n}{n!} pq^{n+1-k} dt \\
&=&pq^{1-k}\displaystyle\sum_{n=k-1}^\infty \frac{(q\lambda )^{n}}{n!} \int_0^\infty e^{-(\lambda +1)t} \ t^ndt \\
&=&  pq^{1-k}\displaystyle\sum_{n=k-1}^\infty \frac{(q\lambda )^{n}}{n!} \frac{\Gamma (n+1)}{(\lambda +1)^{n+1}}   \\
&=&\displaystyle\frac{pq^{1-k}}{\lambda +1}  \displaystyle\sum_{n=k-1}^\infty \left(\frac{q\lambda }{\lambda +1}\right)^n \\
&=&\displaystyle\frac{pq^{1-k}}{\lambda +1} \left(\frac{q\lambda }{\lambda +1}\right)^{k-1} \displaystyle\sum_{j=0}^\infty \left(\frac{q\lambda }{\lambda +1}\right)^j  \\
&=&\displaystyle\frac{p}{1+\lambda p}\left(\frac{\lambda }{\lambda +1 }\right)^ {k-1}.
\end{eqnarray*}

By (\ref{E1:lemaaux1}) we obtain the probability generating function of $Z_B$,
$$\begin{array}{lll}
\phi_B(s)&=&\bbE[s^{Z_B}]=\displaystyle\sum_{k\geq 0}  \mathbb{P}[Z_B=k] \ s^k  \\
&=& \displaystyle\frac{q}{1+\lambda p} + \frac{1+\lambda}{\lambda(1+\lambda p)}\sum_{k\geq 1} \left(\frac{\lambda p s}{1+\lambda p}\right)^{k} \vspace{0,2cm} \\
&=&\displaystyle\frac{1}{1+\lambda p}\left[q+\frac{(\lambda +1)ps}{1+\lambda p -\lambda ps}\right].
\end{array}$$

Besides, from (\ref{E2:lemaaux1}), we obtain the probability generating function of $Z_G$,
$$\begin{array}{lll}
\phi_G(s)&=&\bbE[s^{Z_G}]=\displaystyle\sum_{k\geq 0}  \mathbb{P}[Z_G=k] \ s^k  \\
&=& \displaystyle\frac{q}{1+\lambda p} + \frac{sp}{1+\lambda p}\sum_{k\geq 1} \left(\frac{\lambda  s}{1+\lambda }\right)^{k-1} \vspace{0,2cm} \\
&=&\displaystyle\frac{1}{1+\lambda p}\left[q+\frac{(\lambda +1)ps}{1+\lambda -\lambda s}\right].
\end{array}$$

Finaly, the desired result follows after we observe that $$\phi_{r,2}(s)=r\phi_G(s)+(1-r)\phi_B(s),$$ 
and computing  $\mathbb{E}[Z_1^{r,2}]=\phi_{r,2}'(1).$\\
 \end{proof}

\begin{lem}
\label{L: disp2}  The probability generating function of $Z_1^{r,3}$ is given by: 

\[ \psi_{r,3}(s)=r\psi_{G}(s)+(1-r)\psi_{B}(s),\] where
{\small  
$$\psi _B(s):=\frac{q}{1+\lambda p}+\frac{m(1+\lambda)}{\lambda}\sum_{k=1}^m  {m \choose k}\left[\frac{-\lambda p s}{m(1+\lambda p)}\right]^k\sum_{j=0}^k {k \choose j}\frac{(-1)^j j^k}{m(1+\lambda p)-\lambda p j},$$} 
{\small $$\psi_G(s):=\frac{q}{1+\lambda p}+\frac{(1+\lambda)ps}{\lambda p +1}\sum_{k=1}^m  {m \choose k}\left[\frac{-\lambda  s}{m(1+\lambda )}\right]^{k-1}\sum_{j=0}^k {k \choose j}\frac{(-1)^{j-1}j^k}{m(1+\lambda )-\lambda j}.$$}  
Furthermore,
 $$\mathbb{E}[Z_1^{r,3}]=\frac{mp(\lambda +1)^2r}{(m+\lambda)(\lambda p +1)}+\frac{mp(\lambda +1)(1-r)}{m+ \lambda  p}.   $$

\end{lem}

\begin{proof} Consider $\mct$ starting from one colony placed at some vertex $x \in \mathcal{G}_m$. 
Besides the quantity already defined $Z_1^{r,3},$ consider also $Z$ the number of individuals
that survived right after the collapse, before they compete for space.

From the definition of $\mct$ it follows that
\begin{equation}\label{E}
\mathbb{P}[Z=j]=r\mathbb{P}[Z_G=j]+(1-r)\mathbb{P}[Z_B=j],
\end{equation} where $Z_B$ and $Z_G$ are the random variables defined in (\ref{E1:lemaaux1}) and (\ref{E2:lemaaux1}), respectively. By other side, for $k\in\{1,\ldots,m\}$ and $j\geq k$, observe that 
$$\mathbb{P}[Z_1^{r,3}=k|Z=j]={m \choose k}\frac{T(j,k)}{m^j}.$$
By the inclusion-exclusion principle, $T(j,k)=\sum_{i=0}^k{k \choose i}(-1)^i(k-i)^j$ is the number of surjective functions
whose domain is a set with $j$ elements and whose codomain is a set with $k$ elements. See  Tucker~\cite{Tucker} p. 319.


Then, for $k\in\{1,\ldots,m\},$
\begin{eqnarray}\label{E4: lemaaux2}
\mathbb{P}[Z_1^{r,3}=k]&=&r\sum_{j=k}^\infty {m \choose k}\frac{T(j,k)}{m^j}\mathbb{P}[Z_G=j] \nonumber\\
&&+(1-r)\sum_{j=k}^\infty {m \choose k}\frac{T(j,k)}{m^j}\mathbb{P}[Z_B=j].
\end{eqnarray}

By (\ref{E1:lemaaux1}), we have that\\

{\small
$\displaystyle\sum_{j=k}^\infty {m \choose k}\frac{T(j,k)}{m^j}\mathbb{P}[Z_B=j]$
\begin{eqnarray}\label{E5: lemaaux2}
&=&{m \choose k} \frac{1+\lambda}{\lambda(\lambda p+1)}\sum_{j=k}^\infty \left[\frac{\lambda p}{m(\lambda p+1)}\right]^{j}T(j,k)\nonumber\\
&=&{m \choose k} \frac{1+\lambda}{\lambda(\lambda p+1)} \left[\frac{\lambda p}{m(\lambda p+1)}\right]^{k}\sum_{j=0}^\infty \left[\frac{\lambda p}{m(\lambda p+1)}\right]^{j}T(j+k,k)\nonumber\\
&=& {m \choose k} \frac{1+\lambda}{\lambda(\lambda p+1)} \left[\frac{\lambda p}{m(\lambda p+1)}\right]^{k}\sum_{j=0}^\infty \left[\frac{\lambda p}{m(\lambda p+1)}\right]^{j}\sum_{i=0}^k {k \choose i}(-1)^i(k-i)^{j+k}\nonumber\\
&=& {m \choose k} \frac{1+\lambda}{\lambda(\lambda p+1)} \left[\frac{\lambda p}{m(\lambda p+1)}\right]^{k}\sum_{i=0}^k {k \choose i}(-1)^i(k-i)^{k}\sum_{j=0}^\infty \left[\frac{\lambda p(k-i)}{m(\lambda p+1)}\right]^{j}\nonumber\\
&=&{m \choose k} \frac{m(1+\lambda)}{\lambda} \left[\frac{\lambda p}{m(\lambda p+1)}\right]^{k}\sum_{i=0}^k {k\choose i}\frac{(-1)^i(k-i)^k}{m(\lambda p+1)-\lambda p (k-i)}.
\end{eqnarray}}

Similarly, by (\ref{E2:lemaaux1}), we have that\\

{\small
$\displaystyle\sum_{j=k}^\infty {m \choose k}\frac{T(j,k)}{m^j}\mathbb{P}[Z_G=j]$
\begin{eqnarray}\label{E6: lemaaux2}
&=&{m \choose k} \frac{p}{m(\lambda p+1)}\sum_{j=k}^\infty \left[\frac{\lambda}{m(\lambda+1)}\right]^{j-1}T(j,k)\nonumber\\
&=&{m \choose k} \frac{p}{m(\lambda p+1)}\left[\frac{\lambda}{m(\lambda+1)}\right]^{k-1}\sum_{j=0}^\infty \left[\frac{\lambda}{m(\lambda+1)}\right]^{j}T(j+k,k)\nonumber\\
&=&{m \choose k} \frac{p}{m(\lambda p+1)}\left[\frac{\lambda}{m(\lambda+1)}\right]^{k-1}\sum_{j=0}^\infty \left[\frac{\lambda}{m(\lambda+1)}\right]^{j}\sum_{i=0}^k {k \choose i}(-1)^i(k-i)^{j+k}\nonumber\\
&=&{m \choose k} \frac{p}{m(\lambda p+1)}\left[\frac{\lambda}{m(\lambda+1)}\right]^{k-1}\sum_{i=0}^k {k \choose i}(-1)^i(k-i)^{k}\sum_{j=0}^\infty \left[\frac{\lambda(k-i)}{m(\lambda+1)}\right]^{j}\nonumber\\
&=&{m \choose k}\frac{(1+\lambda)p}{\lambda p +1}\left[\frac{\lambda  }{m(1+\lambda )}\right]^{k-1}\sum_{i=0}^k {k \choose i}\frac{(-1)^i (k-i)^k}{m(1+\lambda )-\lambda(k-i)}.
\end{eqnarray}}

Finally, observe that $\mathbb{P}[Z_1^{r,3}=0]=\mathbb{P}[Z=0]=q/(1+\lambda p)$. With (\ref{E4: lemaaux2}),(\ref{E5: lemaaux2}) and (\ref{E6: lemaaux2}) 
we obtain the probability generating function of $Z_1^{r,3}$.\\

To compute $\mathbb{E}[Z_1^{r,3}],$ consider enumerating each neighbour of 
the initial vertex $x$, from 1 to  $m$. Next we describe  
$Z_1^{r,3}=\sum_{i=1}^m I_{i},$ where $I_{i}$ is the indicator function of the event \{A new colony is created in the first generation at the $i-th$ neighbour vertex of $x$ \}. Therefore, 
\begin{eqnarray}\label{E1: lemaaux2}
\mathbb{E}[Z_1^{r,3}]=\sum_{i=1}^m \mathbb{P}[I_{i}=1]=m\mathbb{P}[I_{1}=1].
\end{eqnarray}

Observe that
\begin{eqnarray}
\mathbb{P}[I_{1}=1|Z=k]
=1-\left(\frac{m-1}{m}\right)^k \nonumber
\end{eqnarray} 
and that by using (\ref{E}) we have that
\begin{eqnarray}\label{E2: lemaaux2}
\mathbb{P}[I_{1}=1]&=&r\sum_{k=1}^\infty\left[1-\left(\frac{m-1}{m}\right)^k\right]\mathbb{P}[Z_G=k] \nonumber \\
&&+(1-r)\sum_{k=1}^\infty\left[1-\left(\frac{m-1}{m}\right)^k\right]\mathbb{P}[Z_B=k]. 
\end{eqnarray}

Substituting (\ref{E1:lemaaux1}) and (\ref{E2:lemaaux1}) in (\ref{E2: lemaaux2}) one can see that
\begin{eqnarray}\label{E3: lemaaux2}
\mathbb{P}[I_1=1]&=&\frac{p(\lambda +1)^2r}{(m+\lambda)(\lambda p +1)}+\frac{p(\lambda +1)(1-r)}{m+ \lambda  p}.
\end{eqnarray}



Finally, plugging (\ref{E3: lemaaux2}) into (\ref{E1: lemaaux2}) we obtain the desired result.
\end{proof}

\begin{proof}[Proofs of Theorems \ref{th:disp1} and \ref{th:dispesp1}]
From Remark~\ref{AuxPro} one can see that $\mci$ survives if and only if $\mathbb{E}[Z_n^{r,i}]>1.$  
From Lemmas~\ref{L:disp} and \ref{L: disp2} the result follows.
\end{proof}

\begin{proof}[Proofs of Theorems \ref{th:disp2} and \ref{th:dispesp2}]
From Remark~\ref{AuxPro} we have that the probabilities of extinction, $\rho_2(r)$ and $\rho_3(r)$, of 
$\mcd$ and $\mct$, are the smallest
solution in $[0,1]$ of $\phi_{r,i}(s)=s$  for $i=2$ and $3$ respectively.  The desired results follow from Lemmas~ \ref{L:disp} and \ref{L: disp2}. 
\end{proof}

\begin{proof}[Proof of Proposition \ref{prop-disp-est} $(i)$]
First we define the following functions
\begin{eqnarray*}
f_m(\lambda)&:=&\frac{mp(1+\lambda)^2r}{(m+\lambda)(\lambda p +1)}+\frac{mp(1+\lambda)(1-r)}{m+ \lambda  p} ,\\ 
f(\lambda)&:=&\frac{p(\lambda+1)^2r}{\lambda p+1} +p(\lambda+1)(1-r).
\end{eqnarray*} 
From Theorems \ref{th:disp1} and \ref{th:dispesp1} it follows that 
$$\ld=\inf\{\lambda: f(\lambda)>1\},$$
$$\lt=\inf\{\lambda: f_m(\lambda)>1\}.$$

Observe that $f_m$ and $f$ are continuous functions on $[0,\infty),$ such that $f_m(0)=f(0)=p<1,$ $\displaystyle\lim_{\lambda\rightarrow \infty}f(\lambda)=\infty$ and $\displaystyle\lim_{\lambda\rightarrow \infty}f_m(\lambda)=m.$\\

Moreover, $\{f_m\}_{m \geq 1}$ is a strictly increasing sequence of strictly increasing functions on $(0,\infty)$ such that $\displaystyle\lim_{m\rightarrow\infty} f_m(\lambda)=f(\lambda)$. Similarly, $f$ is a
strictly increasing function.

Then, from the intermediate value theorem and the strict monotonicity of $f$ we have that there is a unique $\lambda_*\in(0,\infty)$ such that $f(\lambda_*)=1.$ Moreover, from the definition of $\ld$ and the continuity of $f$, we have that
\begin{eqnarray}\label{valorcritico}
f(\lambda)=1  &\iff& \lambda = \ld.
\end{eqnarray}  
Thus, $\lambda_*= \ld.$ Similarly, for $m\geq 2,$ we obtain that 
\[\lt\in(0,\infty)\]
and
\begin{eqnarray}
 f_m(\lambda)=1  &\iff& \lambda = \lt. \nonumber
\end{eqnarray} 
Besides, from the strict monotonicity of $f_1$, it follows that 
\[\lambda^3(p,r,1)=\infty.\]


In order to show that $\lt > \lambda^3(p,r,m+1)$ for all $m \geq 2$ let us 
assume that $\lt \leq \lambda^3(p,r,m+1)$ for some $m \geq 2$ and 
proceed by contradiction. Note that 
\[ 1=f_m(\lt)\leq f_m(\lambda^3(p,r,m+1))<f_{m+1}(\lambda^3(p,r,m+1))=1 \] 
\noindent
which is cleary a contradiction.
Analogously one can show that \[\ld<\lt \textrm{ for all } m\geq 1.\]
\end{proof}

\begin{proof}[Proof of Proposition \ref{prop-disp-est} $(ii)$]
Let us restrict the domain of the functions $f_m$ and $f$ to $[0,\lambda^3(p,r,2)].$ Observe that $f_m$ and $f$ are continuous functions, that $\displaystyle\lim_{m\rightarrow\infty}f_m=f$ and that $f_m(\lambda)<f_{m+1}(\lambda)$ for all $\lambda \in [0,\lambda^3(p,r,2)].$  Then, from Theorem 7.13 in Rudin~\cite{Rudin} we have that
$f_m$ converges uniformly to $f$ on $[0,\lambda^3(p,r,2)]$.

From $(i)$ it follows that $\lt \in [0,\lambda^3(p,r,2)]$ for all $m\geq 2$ and
the existence of $\theta:=\displaystyle\lim_{m\rightarrow\infty} \lt.$ Then, from
the uniform convergence of $f_m$ to $f$, it follows that $f(\theta)=\displaystyle\lim_{m\rightarrow\infty} f_{m}(\lt)=1,$ (see Rudin \cite[exercise 9, chapter 7]{Rudin}). Finaly the result follows from (\ref{valorcritico}). 
\end{proof}

\section{Acknowledgments} The authors are thankful to Rinaldo Schinazi and Elcio Lebensztayn
for helpful discussions about the model. V. Junior and A. Rold\'an wish to thank the Instituto de 
Matem\'atica e Estat\'{\i}stica of Universidade de S\~ao Paulo for the warm hospitality during 
their scientific visits to that institute. The authors are thankful for the two anonymous referees 
for a careful reading and many suggestions and corrections that greatly helped to improve the 
paper.


\begin{thebibliography}{99}

\bibitem{AEL2007} {J.R.Artalejo, A.Economou and M.J.Lopez-Herrero.}  
Evaluating growth measures in an immigration process subject to binomial and geometric catastrophes.
\textit{Mathematical Biosciences and Engineering} \textbf{4}, (4), 573 - 594 (2007).

\bibitem{BGR1982}{P.J.Brockwell, J.Gani and S.I.Resnick.}
Birth, immigration and catastrophe processes.
Adv. Appl. Prob. 14, 709-731 (1982).

\bibitem{CairnsPollett}{B.Cairns and P.K. Pollet.}
Evaluating Persistence Times in Populations that are
Subject to Local Catastrophes
in “MODSIM 2003 International Congress on Modelling and Simulation”
(ed. D.A. Post), Modelling and Simulation Society of Australia and New Zealand, 747–752 (2003).



\bibitem{FMM1995} {G.Fayolle, V.A.Malyshev and M.V.Menshikov.}
Topics in the Constructive Theory of Countable Markov Chains. \textit{Cambridge University Press} (1995) .

\bibitem{KPR2016} {S.Kapodistria, T. Phung-Duc and J. Resing.} 
Linear birth/immigration-death process with binomial catastrophes.
\textit{Probability in the Engineering and Informational Sciences} \textbf{30} (1), 79-111 (2016).

\bibitem{Lanchier}{N.Lanchier.}
Contact process with destruction of cubes and hyperplanes: forest fires
versus tornadoes. 
\textit{J. Appl. Probab} \textbf{48}, 352-365 (2011). 

\bibitem{MRS2015}{F.P.Machado, A. Rold\'an-Correa and R.Schinazi.} 
Colonization and Collapse. \textit{arXiv:1510.02704} (2015). 

\bibitem{Rudin} {W. Rudin.} 
Principles of Mathematical Analysis, Third Edition. 
\textit{McGraw-Hill,Inc.}  (1976).

\bibitem{S2014} {R.Schinazi.} 
Does random dispersion help survival?
\textit{Journal of Statistical Physics}, \textbf{159}, (1), 101-107 (2015).

\bibitem{Tucker}{A.Tucker.} 
Applied Combinatorics 6\textsuperscript{th} ed.
\textit{John Wiley \& Sons, Inc}. (2012).

\end{thebibliography}
\end{document}